\documentclass[runningheads,envcountsame,a4paper]{llncs} 

\usepackage{amssymb,amsmath}
\usepackage{wasysym}
\usepackage{graphicx}
\usepackage{epsfig}
\usepackage{latexsym}
\usepackage[english]{babel}
\usepackage[utf8]{inputenc}
\usepackage{csquotes}
\usepackage{algorithm} 
\usepackage{algpseudocode}
\usepackage{tikz}
\usetikzlibrary{arrows,shapes.geometric,positioning,calc}
\usepackage[colorinlistoftodos]{todonotes}
\usepackage{changepage}
\usepackage{subfig}
\usepackage{soul}
\usepackage{hyperref}
\usepackage{todonotes}
\usepackage{enumerate}

\newtheorem{thm}{Theorem}
\newtheorem{prop}[thm]{Proposition}
\newtheorem{cor}[thm]{Corollary}
\newtheorem{lem}[thm]{Lemma}

\newtheorem{exa}[thm]{Example}

\newtheorem{rem}[thm]{Remark}

\newsavebox{\qedB}

\sbox{\qedB}{\setlength{\unitlength}{1mm}
 \begin{picture}(4,4)(0,0)
  \thinlines
  {\put(0,0){\framebox(2.83,2.83){}}}%
  {\put(1.17,1.17){\framebox(2.83,2.83){}}}%
  {\put(0,0){\framebox(4,4){}}}%
  {\put(1.17,1.17){{\rule{1ex}{1ex} }}}%
 \end{picture}}
 
\newcommand{\QEDB}{\ifmmode\def\next{\tag"\usebox{\qedB}"}%
 \else\let\next=\relax
 {\unskip\nobreak\hfil\penalty50
 \hskip2em\hbox{}\nobreak\hfil\usebox{\qedB}
 \parfillskip=0pt \finalhyphendemerits=0\penalty-100\bigskip}\fi\next}

\newcommand\mbx[1]{\makebox[10pt]{#1}}
\newcommand{\Alphabet}{\hbox{\rm Alph}}

\newcommand{\fac}{\hbox{\rm Fac}}

\newcommand{\prim}{\hbox{\rm Prim}}

\newcommand{\bprop}{\begin{prop}}
\newcommand{\eprop}{\end{prop}}
\newcommand{\bcor}{\begin{cor}}
\newcommand{\ecor}{\end{cor}}
\newcommand{\blem}{\begin{lem}}
\newcommand{\elem}{\end{lem}}

\title{A note on the Lie complexity and beyond}

\author{Shuo Li\inst{1}}

\institute{Laboratoire de Combinatoire et d'Informatique Mathématique,\\
Université du  Québec \`a Montréal,\\
CP 8888 Succ. Centre-ville, Montréal (QC) Canada H3C 3P8\\
\email{shuo.li.ismin@gmail.com}}

\begin{document}

\maketitle


\begin{abstract}
In a recent paper, Jason P. Bell and Jeffrey Shallit introduced the notion of {\em Lie complexity} and proved that the Lie complexity function of an automatic sequence is automatic. In this note, we give more facts concerning Lie complexity and define the extended of Lie complexity and the prefix Lie complexity. Further, we prove that some proprieties of Lie complexity also hold for the extended Lie complexity. Particularly, we prove that the extended Lie complexity function and the first-order difference sequence of the prefix Lie complexity function of an automatic sequence are both automatic.
\end{abstract}

\section{Introduction}

The study of patterns in a word is one of the fundamental topics in combinatorics on words. To understand the behavior of factors of a special form (e.g., palindromes, bordered,
unbordered, squarefree, repetition-free, k-power, etc.) in a word, people introduced and studied various complexity functions. Among these functions, we can mention cyclic complexity~\cite{CASSAIGNE201736}, arithmetical complexity~\cite{AFF}, abelian complexity~\cite{KARHUMAKI1980155}, Lempel-Ziv complexity~\cite{lemz}, and the factor complexity~\cite{allouche_shallit_2003}[Chapt. 10].

In a recent paper~\cite{BellShallit}, Bell and Shallit introduced the notion of {\em Lie complexity}. Given a finite alphabet $\sum$ and an infinite word $w$ over $\sum$, the Lie complexity function $L_w: \mathbb{N} \to \mathbb{N}$ satisfies that $L_w(n)$ is the number of conjugacy classes (under cyclic shift) of length-$n$ factors $u$ of $w$ with the
property that every element of the conjugacy class appears in $w$. From their initial paper, this notion was motivated by ideas from the theory of Lie algebras. In this article, we propose to review this notion from another perspective: the cycles on the Rauzy graphs of the infinite word $w$. In section~\ref{sec:prel} we recall the basic terminology about words. In section~\ref{sec:rauzy} we recall the notion of Rauzy graph, and give an alternative proof of the following theorem announced in~\cite{BellShallit}:

\begin{thm}[Theorem 1.1 in \cite{BellShallit}]
\label{Lie:1}
For every infinite word $w$ over a finite alphabet and for every positive integer $n$, we have $$L_w(n) \leq C_w(n)-C_w(n-1)+1,$$ where $L_w(n)$ is the value of the Lie complexity function of $w$ at $n$ and $C_w(n)$ is the number of length-$n$ factors of $w$. 
\end{thm}

We then define the {\em quasi-small circuits} on Rauzy graph and prove the following theorem:

\begin{thm}
\label{bound}
For every infinite word $w$ over a finite alphabet and for every positive integer $n$, we have  $$Qs_w(n) \leq C_w(n+1)-C_w(n)+1,$$

where $Qs_w(n)$ is the number of quasi-small circuits in the $n$-th Rauzy graph of $w$.
\end{thm}

In section~\ref{sec:elc} we define the extended Lie complexity function and the prefix Lie complexity function of an infinite word $w$. We prove that Theorem~\ref{Lie:1} also holds for the extended Lie complexity functions.

In Section~\ref{sec:facts} we prove more facts concerning Lie complexity, two results are announced as follows: 

\begin{thm}
\label{Lie:2}
For every infinite word $w$ over a finite alphabet and for every positive integer $n$, we have $$L_w(n)+L_w(n+1) \leq C_w(n)-C_w(n-1)+1+|\Alphabet(w)|+p(n+1),$$ where $L_w(k)$ is the value of the Lie complexity function of $w$ at $k$, $C_w(k)$ is the number of length-$k$ factors of $w$, $|\Alphabet(w)|$ is the cardinality of the alphabet of $w$ and $p(k)$ is the cardinality of the set\\ $\left\{[p]| [p] \subset \fac(w), |p|=k, \text{$p$ is primitive}\right\}$.
\end{thm}

\begin{thm}
\label{Lie:3}
With the same notation and the same hypothesis on $w$ as above, for every positive integer $n \geq 6$, we have $$L_w(n)+L_w(n+1)+L_w(n+2) \leq C_w(n)-C_w(n-1)+1+3|\Alphabet(w)|+|\Alphabet(w)|^2+p(n+1)+p(n+2),$$ 
\end{thm}

In Section~\ref{sec:automatic} we prove that the extended Lie complexity function and the first-order difference sequence of the prefix Lie complexity function of a $k$-automatic sequence are still $k$-automatic.
\section{Preliminaries}
\label{sec:prel}

Let $\sum$ be a finite alphabet and $w$ be either a finite word or an infinite word over $\sum$. A {\em factor} of $w$ is a finite block of contiguous symbols occurring within $w$. Particularly, the empty word is also a factor of $w$. The length of the empty word is $0$ and it will be denoted by $\varepsilon$. For any non-negative integer $n$, let $F_w(n)$ denote the set of all length-$n$ factors of $w$ and let $C_w(n)$ denote the cardinality of $F_w(n)$. The sequence $(C_w(n))_{n \in \mathbb{N}}$ is called the {\em factor complexity sequence} of $w$. We let Fac(w) denote
the collection of all factors of $w$ (including the empty word). 

For any finite (resp. infinite) word $w$, a finite word $a$ is called a {\em prefix} of $w$ if there exists another finite (resp. infinite) word $b$ such that $w$ is the concatenation of $a$ and $b$, i.e. $w=ab$. Similarly, for any finite word $w$, a finite word $a$ is called a {\em suffix} of $w$ if there exists a finite word $b$ such that $w=ba$. 

Let $w$ be an infinite word. For any non-negative integer $n$, let $w_n$ denote the length-$n$ prefix of $w$ and let $w[n]$ denote the $n$-th letter in $w$. Thus, $w=w[1]w[2]\cdots w[n] \cdots$ and, for any non-negative integer $n$, $w_n=w[1]w[2]\cdots w[n]$. For any couple of integers $i,j$ with $i \leq j$, let $w[i..j]$ denote $w[i]w[i+1]\cdots w[j]$.

Let $u$ and $v$ be two finite words, they are called {\em conjugate} when there exist words $x,y$ such that $u=xy$ and $v=yx$.
The conjugacy class of a word $w$ is denoted by $[w]$. By convention, let us define $[\varepsilon]=\left\{\varepsilon\right\}$.

For any natural number $k$, we define the {\em $k$-power} of a finite word $u$ to be the concatenation of $k$ copies of $u$, and it is denoted by $u^k$. Particularly, $u^0=\varepsilon$ for any $u$.
A finite word $w$ is said to be {\em primitive} if it is not a power of another word. Let $\prim(w)$ denote the set of primitive factors of $w$. For any finite word $u$ and any positive rational number $\alpha$, the $\alpha$-power of $u$ is defined to be $u^au_0$ where $u_0$ is a prefix of $u$, $a$ is the integer part of $\alpha$, and $|u^au_0|=\alpha |u|$. The $\alpha$-power of $u$ is denoted by $u^{\alpha}$. We also define the $\omega$-power of $u$ as the word $u^\omega = u u u \cdots$. A finite word $a$ is said to be {\em of the period $b$} if there exists a rational number $\alpha$ satisfying $\alpha \geq 1$ and $a=b^{\alpha}$. 

Let $w$ be a word and let $n$ be a positive integer, we define $[w]_n=\left\{u^{\frac{n}{|u|}}|u \in [w] \right\}$. 

\begin{exa}
Let $u=aba$, then $u^{\frac{5}{3}}=abaab$, $u^{\frac{2}{3}}=ab$ and $[u]=\left\{aab,aba,baa \right\}$. Further, $[u]_1=\left\{v^{\frac{1}{3}}|v \in [u] \right\}=\left\{a,b \right\}$. Similarly, $[u]_2=\left\{aa,ab,ba \right\}$ and $[u]_5=\left\{aabaa,abaab,baaba \right\}$. \qed
\end{exa}

Let $w$ be an infinite word over a finite alphabet, the Lie complexity function $L_w: \mathbb{N} \to \mathbb{N}$ is defined as follows: for any positive integer $n$, $L_w(n)$ is the cardinality of the following set
$$\left\{[p]| [p] \subset \fac(w), |p|=n\right\}.$$
Let us define $CL_w(n)=\left\{[p]|[p] \subset \fac(w), |p|=n\right\}$.

\section{Rauzy graphs and quasi-small circuits}
\label{sec:rauzy}
Let us first recall the notion of Rauzy graph. 
Let $w$ be an infinite word over a finite alphabet $\sum$. For any positive integer $n$, let the Rauzy graph $\Gamma_n(w)$ be an oriented graph whose vertex set is $F_w(n)$ and its edge set is $F_{w}(n+1)$;
an edge $e \in F_{w}(n+1)$ starts at the vertex $u$ and ends at the vertex $v$, if $u$ is a
prefix and $v$ is a suffix of $e$. Let us define $\Gamma(w)=\cup_{n=1}^{\infty}\Gamma_n(w)$.

Let $G$ be an oriented graph and let $V$ and $E$ be respectively the vertex set and the edge set of $G$. $G$ is called {\em weakly connected} if for any couple of different vertices $a,b \in V$, there exists a sequence of vertices $a=v_1,v_2,...v_k=b$ such that for any integer $i$ satisfying $1 \leq i \leq k-1$, there exists an edge $e_i$ either from $v_i$ to $v_{i+1}$ or from $v_{i+1}$ to $v_i$. It is well known that the Rauzy graphs of any words are weakly connected.

Let $\Gamma_l(w)$ be a Rauzy graph of $w$. A sub-graph in $\Gamma_l(w)$ is called a {\em simple closed path} if there are $j$ vertices $v_1,v_2,\dots, v_j$ and $j$ distinct edges $e_1,e_2,\dots,e_j$ for some integer $j$, such that for each $t$ with $1 \leq t \leq j-1$, the edge $e_t$ starts at $v_t$ and ends at $v_{t+1}$, and for the edge $e_j$, it starts at $v_j$ and ends at $v_1$. A simple closed path is called a {\em elementary circuit} if the $j$ vertices in the path are distinct; further, $j$ is called the {\em size} of the circuit. 

\begin{proof}[ of Theorem~\ref{Lie:1}]

This theorem is trivial for $n=1$.

Let us suppose $n \geq 2$. For any class of conjugacy $[p]$ in $CL_w(n)$, there exists a closed simple path in $\Gamma_w(n-1)$ such that its edge set is $[p]$. Further, the edge sets of these closed simple paths are pairwise disjoint. Now let us remove one edge in each of these closed simple paths, we then have a sub-graph graph $\Gamma'_w(n-1)$. This graph is still weakly connected. However, for a weakly connected graph, we have $e-v+1 \geq 0$, where $e$ is the number of edges and $v$ is the number of vertices in the graph. Applying this relation on $\Gamma'_w(n-1)$, we have the number of vertices is $ C_w(n)-L_w(n)$ and the number of edges is $C_w(n-1)$, thus $ C_w(n)-L_w(n)-C_w(n-1)+1 \geq 0$. \qed
\end{proof}

\begin{rem}
We can prove that the simple closed path with the edge set $[p]$ is, in fact, a simple circuit, which will be proved in Lemma~\ref{bijection}. However, as it is not mandatory in this proof, we will prove this fact in the following of the article. \qed
\end{rem}

Now let us give some basic proprieties of the circuits in the Rauzy graphs.

\begin{lem}
\label{cycles}
Let $w$ be either a finite word or an infinite word over a finite alphabet and let $\Gamma_l(w)$ be a Rauzy graph of $w$ for some $l$ satisfying $1 \leq l \leq |w|$. Then for any elementary circuit $C$ in $\Gamma_l(w)$, there exists a unique primitive word $q$, up to conjugacy, such that the vertex set of $C$ is $\left\{p^{\frac{l}{|p|}}| p \in [q]\right\}$ and its edge set is $\left\{p^{\frac{l+1}{|p|}}| p \in [q]\right\}$.
\end{lem}

\begin{proof}
In the case that the size of $C$ is no larger than $l$, the lemma is proved by Lemma 5 in~\cite{Brlekli}. 
If the size of $C$ is larger than $l$, let its size be $k$ and let its vertices and edges be respectively $v_1,v_2,...,v_k$ and $e_1,e_2,...,e_k$. For each $i$ satisfying $1 \leq i \leq k$, let us define $p_i$ to be a word by concatenating consecutively the last letter of words $e_i,e_{i+1},...,e_{i+k-1}$ with $e_{r}=e_{r-k}$ if $r\geq k+1$. The words $p_1,p_2,...,p_k$ are pairwise conjugate. Further, for each $i$ satisfying $1 \leq i \leq k$, from the fact that the edges in the order of $e_i,e_{i+1},...,e_{i+k-1}$ form a circuit, we can deduce that $v_i$ is a suffix of $v_ip_i$. From the hypothesis that $|p_i| > |v_i|$, we have $v_i$ is the length-$k$ suffix of $p_i$ and $e_i$ is the length-$k+1$ suffix of $p_i$. As $p_1,p_2,...,p_k$ are pairwise conjugate, the vertex set of $C$ is $\left\{p^{\frac{k}{|p|}}| p \in [p_1]\right\}$ and its edge set is $\left\{p^{\frac{k+1}{|p|}}| p \in [p_1]\right\}$.

Here we prove the primitivity of $p_1$. If $p_1$ is not primitive, then there are two distinct integers $r,s$ satisfying $1 \leq r,s \leq k$ such that $p'_r=p'_s$, and further, $v_r=v_s$. This contradicts the hypothesis that $C$ contains $k$ distinct vertices.\qed
\end{proof}

From the previous lemma, each elementary circuit can be identified by an associated primitive word $q$ defined in the previous lemma and an integer $l$ such that $\Gamma_l(w)$ is the Rauzy graph in which the circuit is located. Let all elementary circuit be denoted by $C(q,l)$ with the parameters defined as above.

In~\cite{Brlekli} we defined the notion of {\em small circuit}, in this note we extend this notion to {\em quasi-small circuit}. Recall that the {\em small circuits} in the graph $\Gamma_l(w)$ are those elementary circuits whose sizes are no larger than $l$. A circuit $C$ in $\Gamma_l(w)$ is defined to be {\em primitive} if it is elementary and there exists a primitive factor $u$ of length $l+1$ such that the vertex set of $C$ is $[u]_{l}$ and the edge set is $[u]_{l+1}$. A circuit is called {\em quasi-small} if it is either small or primitive. Consequently, if $C(q,l)$ is a quasi-small circuit in $\Gamma_l(w)$, then $l \geq |q|-1$.

\begin{proof}[of Theorem~\ref{bound}]
We remark that the statements previously announced for small circuits of a finite word in Lemma 7 and Lemma 8 in~\cite{Brlekli} also hold for quasi-small circuits of a finite word as well as an infinite word over a finite alphabet. Thus, Theorem~\ref{bound} can be proved by using the arguments given in Lemma 10 and Lemma 11 in~\cite{Brlekli}. \qed
\end{proof}

\begin{exa}
Let us define $u=abaaabaaaaba$, the Rauzy graph $\Gamma_4(u^{\omega})$ is as follows:
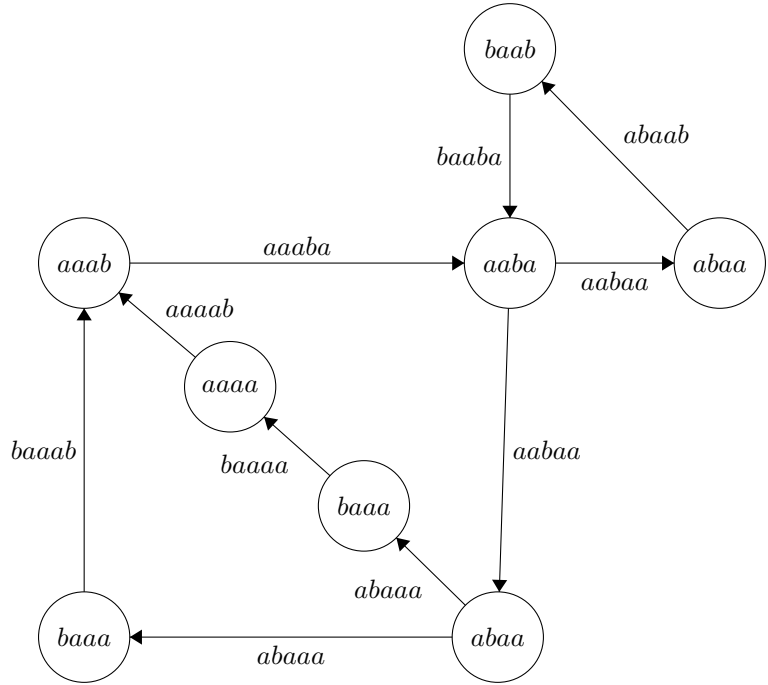
\begin{center}
\begin{tikzpicture}[scale=0.2]
\tikzstyle{every node}+=[inner sep=0pt]
\draw [black] (15,-25.6) circle (3);
\draw (15,-25.6) node {$aaaa$};
\draw [black] (33.4,-17.4) circle (3);
\draw (33.4,-17.4) node {$aaba$};
\draw [black] (23.8,-33.5) circle (3);
\draw (23.8,-33.5) node {$baaa$};
\draw [black] (32.6,-42.2) circle (3);
\draw (32.6,-42.2) node {$abaa$};
\draw [black] (5.4,-17.4) circle (3);
\draw (5.4,-17.4) node {$aaab$};
\draw [black] (5.4,-42.2) circle (3);
\draw (5.4,-42.2) node {$baaa$};
\draw [black] (33.4,-3.2) circle (3);
\draw (33.4,-3.2) node {$baab$};
\draw [black] (47.2,-17.4) circle (3);
\draw (47.2,-17.4) node {$abaa$};
\draw [black] (5.4,-39.2) -- (5.4,-20.4);
\fill [black] (5.4,-20.4) -- (4.9,-21.2) -- (5.9,-21.2);
\draw (4.9,-29.8) node [left] {$baaab$};
\draw [black] (29.6,-42.2) -- (8.4,-42.2);
\fill [black] (8.4,-42.2) -- (9.2,-42.7) -- (9.2,-41.7);
\draw (19,-42.7) node [below] {$abaaa$};
\draw [black] (33.3,-20.4) -- (32.7,-39.2);
\fill [black] (32.7,-39.2) -- (33.22,-38.42) -- (32.22,-38.39);
\draw (33.55,-29.81) node [right] {$aabaa$};
\draw [black] (30.47,-40.09) -- (25.93,-35.61);
\fill [black] (25.93,-35.61) -- (26.15,-36.53) -- (26.85,-35.82);
\draw (25.4,-38.33) node [below] {$abaaa$};
\draw [black] (21.57,-31.5) -- (17.23,-27.6);
\fill [black] (17.23,-27.6) -- (17.49,-28.51) -- (18.16,-27.77);
\draw (16.61,-30.04) node [below] {$baaaa$};
\draw [black] (36.4,-17.4) -- (44.2,-17.4);
\fill [black] (44.2,-17.4) -- (43.4,-16.9) -- (43.4,-17.9);
\draw (40.3,-17.9) node [below] {$aabaa$};
\draw [black] (45.11,-15.25) -- (35.49,-5.35);
\fill [black] (35.49,-5.35) -- (35.69,-6.27) -- (36.41,-5.58);
\draw (40.83,-8.83) node [right] {$abaab$};
\draw [black] (33.4,-6.2) -- (33.4,-14.4);
\fill [black] (33.4,-14.4) -- (33.9,-13.6) -- (32.9,-13.6);
\draw (32.9,-10.3) node [left] {$baaba$};
\draw [black] (12.72,-23.65) -- (7.68,-19.35);
\fill [black] (7.68,-19.35) -- (7.96,-20.25) -- (8.61,-19.49);
\draw (12.99,-21.01) node [above] {$aaaab$};
\draw [black] (8.4,-17.4) -- (30.4,-17.4);
\fill [black] (30.4,-17.4) -- (29.6,-16.9) -- (29.6,-17.9);
\draw (19.4,-16.9) node [above] {$aaaba$};
\end{tikzpicture}
\end{center}

In this graph, there are three elementary circuits: $C(aaaab,4)$, $C(aaab,4)$ and $C(aab,4)$. All of these circuits are quasi-small: $C(aaab,4)$ and $C(aab,4)$ are both small, while $C(aaaab,4)$ is not small, it is primitive. Remark that the number of vertices and edges in this graph are respectively $8$ and $10$, the total number of quasi-small cycles in this graph, which is $3$, is bounded by $10-8+1$.\qed
\end{exa}

\section{Extended Lie complexity and prefix Lie complexity}
\label{sec:elc}
In this section, we consider two different extensions of Lie complexity.
For a given infinite word $w$ over a finite alphabet, the extended Lie complexity function $eL_w: \mathbb{N} \to \mathbb{N}$  is defined that for any positive integer $n$, $eL_w(n)$ is cardinality of the following set:

$$\left\{[p]_n|[p]_n \subset \fac(w), n \geq |p|, p \in \prim(w) \right\}.$$

It is obvious that for any positive integer $n$, $L_w(n) \leq eL_w(n)$. However, we have the following fact:

\begin{thm}
\label{ELie}
For every infinite word $w$ over a finite alphabet and for every positive integer $n$, we have $$eL_w(n) \leq C_w(n)-C_w(n-1)+1.$$ 
\end{thm}

\begin{exa}
Let us consider the Lie complexity and the extended Lie complexity functions of the word $v=(aba)^{\omega}$:
$$
 \begin{array}{|c|c|c|c|c|c|c|c|c|c|c|c|c|}
  \hline
  n & \mbx{0} & \mbx{1} & \mbx{2} & \mbx{3} & \mbx{4} & \mbx{5} & \mbx{6} & \mbx{7} & \mbx{8} & \mbx{9} & \mbx{10} & \cdots \\
  \hline
  L_{v}(n)& 1 & 2 & 2 & 1 & 0 & 0 & 1 & 0 & 0 & 1 & 0 & \cdots\\
  \hline
  eL_v(n) & 1 & 2 & 2 & 1 & 1 & 1 & 1 & 1 & 1 & 1 & 1 & \cdots\\
  \hline
 \end{array}
$$ 
We can easily check that $L_{v}(n) \leq eL_{v}(n)$ for every positive integer $n$.\qed
\end{exa}

Theorem~\ref{ELie} is a consequence of the following lemmas.

\begin{lem}
\label{size}
Let $p,q$ be two primitive words such that $|p|=|q|$, then for any integer $n$ such that $n \geq |p|-1$, the cardinality of $[p]_n$ equals $|p|$; further, $[p]_n = [q]_n$ if and only if $p,q$ are conjugate. 
\end{lem}

\begin{proof}
If $n \geq |p|$, there exists a bijection from $[p]_n$ to $[p]$. Thus, the statement is true.

If $n=|p|-1$, let us first prove that the cardinality of $[p]_n$ equals $|p|$. If it is not the case, there exist a word $t \in [p]_n$ and two different letters $a,b$ such that $ta,tb \in [p]$. However, it cannot be true because the number of occurrences of $a$ in $ta$ and $tb$ are not same.

Now let us suppose $n=|p|-1=|q|-1$ and $[p]_n = [q]_n$. From the hypothesis, there exists a word $t \in [p]_n$ and two letters $a,b$ such that $ta \in [p], tb \in [q]$. However, with the hypothesis $[p]_n = [q]_n$, we can deduce that the numbers of occurrences of each letter in $p$ and in $q$ should be the same. Thus, $a=b$ and $p,q \in [ta]$.\qed
\end{proof}

\begin{lem}
\label{bijection}
Let $w$ be an infinite word with over a finite alphabet and let $n$ be a positive integer. 
For any $[p]_n \in \left\{[p]_n|[p]_n \subset \fac(w), n \geq |p|, p \in \prim(w) \right\}$, there exists a quasi-small cycle $C(p, n-1)$ in the graph $\Gamma_w(n-1)$. Further, let $[p]_n,[q]_n$ be two elements in $\left\{[p]_n|[p]_n \subset \fac(w), n \geq |p| \right\}$, $C(p, n-1)=C(q, n-1)$ if and only if $p,q$ are conjugate.
\end{lem}

\begin{proof}
If $[p]_n \in \left\{[p]_n|[p]_n \subset \fac(w), n \geq |p|, p \in \prim(w) \right\}$, we can easily check that there exists a simply closed path $([p]_{n-1},[p]_n)$ in $\Gamma_w(n-1)$. Further, from the previous lemma, this simply closed path is a quasi-small circuit and can be identified as $C(p, n-1)$.

If there exist two primitive $p,q$ such that $|p| \leq n$, $|q| \leq n$ and $C(p, n-1)=C(q, n-1)$, then $|[p]_n|=|[q]_n|$. Further, from previous lemma, $|p|=|q|$, and moreover, $p,q$ are conjugate.\qed
\end{proof}

\begin{proof}[of Theorem~\ref{ELie}]
This statement is obviously true for $n=1$. 

Now for any positive integer $n \geq 2$, from the previous lemma, there exists an injection from $\left\{[p]_n|[p]_n \subset \fac(w)\right\}$ to quasi-small cycles in $\Gamma_w(n-1)$, we then can conclude by using Theorem~\ref{bound}. \qed
\end{proof}

Here let us define the prefix Lie complexity. Let $w$ be a finite word, the prefix Lie complexity function $pL_w: \mathbb{N} \to \mathbb{N}$ counts the number of conjugacy classes such that every element of the conjugacy class appears in a prefix of $w$. Formally, for any positive integer $i$ satisfying $1 \leq i \leq |w|$, $pL_w(i)$ is the cardinality of the following set: 
$$\left\{[p]| [p] \subset \fac(w_i)\right\}.$$

\begin{thm}
\label{prefix}
Let $w$ be a finite word, then for any integer $i$ satisfying $i \geq 1 $, one has:

$$0 \leq pL_w(i+1)-pL_w(i) \leq 1.$$
\end{thm}

\begin{proof}
First, it is trivial that $$\left\{[p]| [p] \subset \fac(w_i)\right\} \subset \left\{[p]| [p] \subset \fac(w_{i+1})\right\}.$$ Thus, $0 \leq pL_w(i+1)-pL_w(i)$.

Now let us prove $pL_w(i+1)-pL_w(i) \leq 1$ by contradiction. If this inequality does not hold, then there exist two different conjugacy classes $[p], [q]$ such that 
$[p], [q] \subset \fac(w_{i+1})$ and $[p], [q] \not \subset \fac(w_i)$. Further, there exist two words $p', q'$ satisfying the following proprieties:\\
1) $p' \in [p], q' \in [q];$\\
2) the suffix of $w_{i+1}$ of length $|p|$ (resp. of length $|q|$) is the uni-occurrence of $p'$ (resp. $q'$) in $w_{i+1}$.\\
From 1) and 2), $|p'| \neq |q'|$, otherwise, $p'=q'$ and $[p]=[q]$. Without loss of generality, let us suppose that $|p'| < |q'|$. Thus, there exists a non-empty word $t$ such that $q'=tp'$. Remark that $p't \in [q]$, hence, $p't \in \fac (w_{i+1})$. However, if it is the case, $p'$ has an occurrence in $w_i$. This contradicts the propriety 2) listed as above. Hence, for any positive integer $i$ satisfying $i \geq 1$, there exists at most one conjugacy class $[p]$ such that $[p] \subset \fac(w_{i+1})$ and $[p] \not \subset \fac(w_i)$. we conclude.\qed
\end{proof}

Let $w$ be an infinite sequence. Let us define the the sequence $(\Delta pL_w(n))_{n \in \mathbb{N}}$ as follows: For any non-negative integer $n$, let
$$\Delta pL_w(n)=pL_w(n+1)-pL_w(n).$$

\begin{cor}
\label{cor}
For any infinite sequence $w$, the sequence $(\Delta pL_w(n))_{n \in \mathbb{N}}$ is a $0,1$-sequence.
\end{cor}

\section{More facts on Lie complexity}
\label{sec:facts}

In this section we prove Theorem~\ref{Lie:2} and ~\ref{Lie:3}.

\begin{lem}
\label{length-2}
Let $w$ be an infinite word over a finite alphabet and let $n$ be a positive integer larger than $3$, then there exists a bijection between the sets:
$$\cup_{i|n, 1 \leq i <n}\left\{[p]_{n}|[p]_{n} \subset \fac(w), p \in \prim(w), |p|=i\right\},$$
$$\left\{C(p,n-2)|[p]_{n} \subset \fac(w), p \in \prim(w), |p|< n\right\}.$$
\end{lem}

\begin{proof}
From Lemma~\ref{bijection}, there exists a bijection between $$\cup_{i|n, 1 \leq i <n}\left\{[p]_{n}|[p]_{n} \subset \fac(w), p \in \prim(w), |p|=i\right\},$$ and
$$\left\{C(p,n-1)|[p]_{n} \subset \fac(w), p \in \prim(w), |p|< n\right\}.$$ However, if there exists a circuit $C(p,n-1)$ in $\Gamma_{n-1}$ for some primitive $p$ with $|p| \mid m$ and $|p| \neq m$, there exists a circuit $C(p,n-2)$ in $\Gamma_{n-2}$. Thus, there exists the bijection stated as above.
\end{proof}

\begin{lem}
\label{length-3}
Let $w$ be an infinite word over a finite alphabet and let $n$ be a positive integer larger than $5$, then there exists a bijection between the sets:
$$\cup_{i|n, 1 \leq i <n}\left\{[p]_{n}|[p]_{n} \subset \fac(w), p \in \prim(w), |p|=i\right\};$$
$$\left\{C(p,n-3)|[p]_{n} \subset \fac(w), p \in \prim(w), |p|<n\right\}.$$
\end{lem}

\begin{proof}
A similar argument as above proves the bijection.
\end{proof}

\begin{lem}
\label{length-4}
Let $w$ be an infinite word over a finite alphabet and let $n$ be a positive integer, then there exists a bijection between the sets:
$$\left\{[p]_{n}|[p]_{n} \subset \fac(w), p \in \prim(w), |p|=n\right\};$$
$$\left\{C(p,n-1)|[p]_{n} \subset \fac(w), p \in \prim(w), |p|=n\right\}.$$
\end{lem}

\begin{proof}
It is a corollary of Lemma~\ref{bijection}.
\end{proof}

\begin{proof}[of Theorem~\ref{Lie:2}]
Let $w$ be an infinite word over a finite alphabet and let $n$ be a positive integer. If we let $CL_w(n)=\left\{[p]|[p] \subset \fac(w), |p|=n\right\}$, then
$$CL_w(n)=\cup_{i|n}\left\{[p]_{n}|[p]_{n} \subset \fac(w), p \in \prim(w), |p|=i\right\}.$$

Let us prove the theorem for $n=1$. In this case, we only need to prove that

$$L_w(1)+L_w(2)\leq 2|\Alphabet(w)| + p(2).$$
It is obvious because $L_w(1) =|\Alphabet(w)|$ and 
\begin{align*}
L_w(2)&=|\left\{[p]_{2}|[p]_{2} \subset \fac(w), p \in \prim(w), |p|=1\right\}|+|\left\{[p]_{2}|[p]_{2} \subset \fac(w), p \in \prim(w), |p|=2\right\}|\\
&\leq |\Alphabet(w)| + p(2).
\end{align*}

To prove this statement for $n \geq 2$, let us investigate the quasi-small cycles in $\Gamma_w(n-1)$. From Lemma~\ref{length-2} and ~\ref{length-4} there exist bijections:
$$F:\cup_{i|n+1, 1 \leq i <n+1}\left\{[p]_{n+1}|[p]_{n+1} \subset \fac(w), p \in \prim(w), |p|=i\right\} \to S_{n+1},$$
$$G:\cup_{i|n, 1 \leq i \leq n}\left\{[p]_{n}|[p]_{n} \subset \fac(w), p \in \prim(w), |p|=i\right\} \to S_{n};$$
with $S_{n+1}=\left\{C(p, n-1)| p \in \prim(w), |p| | n+1, |p| \neq n+1\right\}$ \\
and $ S_{n}=\left\{C(p, n-1)| p \in \prim(w), |p| | n\right\}$.
Knowing that all elements in $S_{n+1} \cup S_{n}$ are quasi-small cycles in $\Gamma_w(n-1)$. Combining with the following relations:
$$|S_{n+1} \cap S_{n}| \leq |\left\{C(p, n-1)| p \in \prim(w), |p|=1\right\}| \leq |\Alphabet(w)|,$$
$$L_w(n+1)=|CL_w(n+1)|=|\cup_{i|n, i \neq 1}\left\{[p]_{n}|[p]_{n} \subset \fac(w), p \in \prim(w), |p|=i\right\}|+p(n+1),$$
we conclude that 

\begin{align*}
L_w(n+1)+L_w(n)&=|S_{n+1}|+p(n+1)+|S_n|\\
&= |S_{n+1}\cup S_n|+p(n+1)+|S_{n+1}\cap S_n|\\
&\leq C_w(n)-C_w(n-1)+1+|\Alphabet(w)|+p(n+1).
\end{align*}

\end{proof}

\begin{proof}[of Theorem~\ref{Lie:3}]
With the same notation as above and with a $n$ larger than $5$. let us investigate the quasi-small cycles in $\Gamma_w(n-1)$. From Lemma~\ref{length-2},~\ref{length-3} and Lemma~\ref{length-4} there exist bijections:
$$F:\cup_{i|n+1, 1 \leq i <n+1}\left\{[p]_{n+1}|[p]_{n+1} \subset \fac(w), p \in \prim(w), |p|=i\right\} \to S_{n+1},$$
$$G:\cup_{i|n, 1 \leq i \leq n}\left\{[p]_{n}|[p]_{n} \subset \fac(w), p \in \prim(w), |p|=i\right\} \to S_{n};$$
$$H:\cup_{i|n+2, 1 \leq i <n+2}\left\{[p]_{n+2}|[p]_{n+2} \subset \fac(w), p \in \prim(w), |p|=i\right\} \to S_{n+2},$$

with $S_{n+1}=\left\{C(p, n-1)| p \in \prim(w), |p| | n+1, |p| \neq n+1\right\}$, \\
$ S_{n}=\left\{C(p, n-1)| p \in \prim(w), |p| | n\right\}$\\ 
and $S_{n+2}=\left\{C(p, n-1)| p \in \prim(w), |p| | n+2, |p| \neq n+2\right\}$.

Further, all elements in $S_{n+1} \cup S_{n} \cup S_{n+2}$ are quasi-small cycles in $\Gamma_w(n-1)$. With some analogue relations as above:
$$|S_{n+1} \cap S_{n}| \leq |\Alphabet(w)|, |S_{n+2} \cap S_{n+1}| \leq |\Alphabet(w)|,$$
$$|S_{n+2} \cap S_{n}| \leq |\left\{C(p, n-1)| p \in \prim(w), |p|=1,2\right\}| \leq |\Alphabet(w)|+|\Alphabet(w)|^2,$$
$$L_w(n+1)=|CL_w(n+1)|=|\cup_{i|n, i \neq n+1}\left\{[p]_{n}|[p]_{n} \subset \fac(w), p \in \prim(w), |p|=i\right\}|+p(n+1),$$
$$L_w(n+2)=|CL_w(n+2)|=|\cup_{i|n, i \neq n+2}\left\{[p]_{n}|[p]_{n} \subset \fac(w), p \in \prim(w), |p|=i\right\}|+p(n+2),$$
we conclude that 
\begin{align*}
&L_w(n+2)+L_w(n+1)+L_w(n)=|S_{n+2}|+p(n+2)+|S_{n+1}|+p(n+1)+|S_n|\\
&= |S_{n+2} \cup S_{n+1}\cup S_n|+|S_{n+1}\cap S_n|+|S_{n+2}\cap S_n|+|S_{n+2}\cap S_{n+1}|+p(n+1)+p(n+2)\\
&\leq C_w(n)-C_w(n-1)+1+3|\Alphabet(w)|+|\Alphabet(w)|^2+p(n+1)+p(n+2).
\end{align*}

\end{proof}

\section{Automaticity}
\label{sec:automatic}

A sequence $(s_n)_{n \in \mathbb{N}}$ is called $k$-automatic if there exists a finite automaton such that $s_n$ is the output of the this automaton when imputing
the base-$k$ representation of $n$. A sequence $(t_n)_{n \in \mathbb{N}}$ taking values in $\mathbb{Z}$ is $k$-regular if there is a row vector $v$, a column vector $w$, and a matrix-valued morphism $\zeta : \left\{0, 1, . . . , k-1\right\} \to \mathbb{Z}^{d \times d}$ such that $t_n = v\zeta(x)w$, where $x$ is the base-$k$ representation of $n$. From ~\cite{allouche_shallit_2003}[Thm. 16.1.5], a $k$-regular sequence is k-automatic if it takes only finitely
many distinct values. In this section, we prove the following result by using similar arguments given in~\cite{BellShallit}:

\begin{thm}
\label{auto}
Let $s$ be a $k$-automatic sequence, then the sequences $(eL_w(n))_{n \in \mathbb{N}}$ and $(\Delta pL_w(n))_{n \in \mathbb{N}}$ are both $k$-automatic. Moreover, $(pL_w(n))_{n \in \mathbb{N}}$ is $k$-regular.
\end{thm}

The proof of Theorem~\ref{auto} is based on the following result~\cite{10.1007/978-3-642-22321-1_15}:
\begin{thm}
Let $s$ be a $k$-automatic sequence.\\
(a) There is an algorithm that, given a well-formed first-order logical formula $\phi$
in $FO(\mathbb{N}, +, 0, 1, n \to s[n])$ having no free variables, decides if $\phi$ is true or false.\\
(b) Furthermore, if $\phi$ has free variables, then the algorithm constructs an automaton recognizing the representation of the values of those variables for
which $\phi$ evaluates to true.
\end{thm}
Further, from~\cite{10.1007/978-3-642-22321-1_15}, if $A$ is an automaton accepting the base-$k$ representation of pairs $(i, n)$ in parallel, then the sequence $a_n = \#\left\{i : \text{ $A$ accepts $(i, n)$ }\right\}$ is $k$-regular. 

\begin{proof}[of Theorem~\ref{auto}] We first show that the sequences $(eL_w(n))_{n \in \mathbb{N}}$ and $(pL_w(n))_{n \in \mathbb{N}}$ are both $k$-regular. 
Let us construct a first-order logical formula e-lie$(i, n)$ for the pairs $(i, n)$ as follows:\\
(a) There exists an integer $m$ satisfying $1 \leq m \leq n$;\\
(b) the word $s=w[i..i+m-1]$ is primitive and lexicographically least in its conjugacy class;\\
(c) $[s]_n \in \fac(w)$;\\
(d) the word $w[i..i+m-1]$ is the earliest occurrence of $s$ in $w$.\\
Then the number of $i$ making e-lie$(i, n)$ true equals $eL_w(n)$.\\

Similarly, let us construct a first-order logical formula p-lie$(i, n)$ for the pairs $(i, n)$ as follows:\\
(a) There exists an integer $m$ satisfying $1 \leq m \leq n$;\\
(b) the word $s=w[i..i+m-1]$ is lexicographically least in its conjugacy class;\\
(c) $[s] \in \fac(w[1..n])$;\\
(d) the word $w[i..i+m-1]$ is the earliest occurrence of $s$ in $w$.\\
Then the number of $i$ making p-lie$(i, n)$ true equals $pL_w(n)$.\\

We construct these formulas in a number of steps:\\
factoreq$(i, j, n)$ asserts that the length-$n$ factor $w[i..i+n-1]$ equals $w[j..j+
n-1]$.\\
prim$(i, n)$ asserts that the length-$n$ factor $w[i..i+n-1]$ is primitive.\\
power$(i, m,n)$ asserts that the length-$n$ factor $w[i..i+n-1]$ is of the period $w[i..i+m-1]$, i.e. $w[i..i+n-1]=(w[i..i+m-1])^{\frac{n}{m}}$.\\
perfac$(k,j,m,n)$ asserts that $w[j..j+m-1]$ is a period of $w[k..k+n-1]$.\\
shift$(i, j, n, t)$ asserts that $w[i..i + n-1]$ is the shift, by $t$ positions, of the
factor $w[j..j + n-1]$, i.e. there are two words $a$,$b$ such that $|a|=n-t$, $|b|=t$, $w[j..j + n-1]=ab$ and $w[i..i + n-1]=ba$.\\
conj$(i, j, n)$ asserts that the factor $w[i..i + n-1]$ is a conjugate of $w[j..j +
n-1]$.\\
lessthan$(i, j, n)$ asserts that the factor $w[i..i + n-1]$ is lexicographically
smaller than $w[j..j + n-1]$.\\
lessthaneq$(i, j, n)$ asserts that the factor $w[i..i + n - 1]$ is lexicographically
smaller than or equal to the factor $w[j..j + n-1]$.\\
allconj$(i, n)$ asserts that all conjugates of $w[i..i + n- 1]$ appear as factors of
$w$.\\
allconjpref$(i,m, n)$ asserts that all conjugates of $w[i..i + m- 1]$ appear as factors of
$w[1..n]$.\\
allpower$(i,m,n)$ asserts that for all elements $s$ in $[w[i..i + m- 1]]$, $s^{\frac{n}{m}}$ appear as factors of
$w$.\\
lexleast$(i, n)$ asserts that $w[i..i + n - 1]$ is lexicographically least among all
its conjugates that actually appear in $w$.\\
e-lie$(i, n)$ asserts that there exists an integer $m$ such that $1 \leq m \leq n$, that $w[i..i + m - 1]$ is primitive, that all elements in $[w[i..i + m - 1]]_n$ appear in $w$, that
$w[i..i+m-1]$ is the lexicographically least in its conjugacy class and that $w[i..i+m-1]$
is its first occurrence in $w$.\\
p-lie$(i, n)$ asserts that there exists an integer $m$ such that $1 \leq m \leq n$, that all elements in $[w[i..i + m - 1]]$ appear in $w[1..n]$, that
$w[i..i+m-1]$ is the lexicographically least in its conjugacy class and that $w[i..i+m-1]$
is its first occurrence in $w[1..n]$.

Here are the definitions of the formulas. Recall that the domain of all variables
is $\mathbb{N} = \left\{0, 1, \cdots\right\}$.\\
factoreq$(i, j, n) := \forall u, v (i + v = j + u \land u \geq i \land u < i + n) \implies w[u] = w[v]$\\
prim$(i, m):=\lnot(\exists j (j>0 \land j<m \land \text{factoreq}(i,i+j,m-j) \land \text{factoreq}(i,(i+m)-j,j))$\\
power$(i, m,n):= m\geq 1 \land m \leq n \land (\forall t (t+m<n) \implies w[i+t]=w[i+t+m])$\\
perfac$(k,j,m,n):=m\geq 1 \land m \leq n \land \text{power}(k,m,n) \land \text{factoreq}(j,k,m)$\\
shift$(i, j, n, t) := \text{factoreq}(j, i + t, n - t) \land \text{factoreq}(i,(j + n) - t, t)$\\
conj$(i, j, n) := \exists t (t \leq n) \land \text{shift}(i, j, n, t)$\\
lessthan$(i, j, n) :=\exists t (t < n) \land \text{factoreq}(i, j, t) \land w[i + t] < w[j + t]$\\
lessthaneq$(i, j, n) := \text{lessthan}(i, j, n) \lor \text{factoreq}(i, j, n)$\\
allconj $(i, n) := \forall t(t \leq n) \implies \exists j \text{shift}(i, j, n, t)$\\
allconjpref $(i,m, n) := \forall t(t \leq m) \implies \exists j(j+m \leq n) \land \text{shift}(i, j, m, t)$\\
allpower$(i,m,n):=\forall j \text{conj}(i,j,m) \implies \exists k \text{perfac}(j,k,m,n)$\\
lexleast $(i, n) := \forall j conj(i, j, n) \implies lessthaneq(i, j, n)$\\
e-lie$(i, n) := \exists m (1\leq m \land m\leq n ) \land \text{prim}(i,m) \land \text{allpower}(i, m,n) \land \text{lexleast}(i, m) \land (\forall j \text{factoreq}(i, j, m) \implies (j \geq i))$\\
p-lie$(i,n):= \exists m (1\leq m \land m\leq n ) \land \text{allconjpref}(i,m,n) \land \text{lexleast}(i,m) \land (\forall j \text{factoreq}(i,j,m) \implies j\geq i)$\\

From the remarks preceding the proof, we prove that $(eL_w(n))_{n \in \mathbb{N}}$ and $(pL_w(n))_{n \in \mathbb{N}}$ are both $k$-regular. Since the factor complexity of automatic sequences are linear~\cite{allouche_shallit_2003}[Thm. 10.3.1], and, from Theorem~\ref{ELie}, $(eL_w(n))_{n \in \mathbb{N}}$ is bounded, the sequence $(eL_w(n))_{n \in \mathbb{N}}$ is hence automatic. Further, as $(pL_w(n))_{n \in \mathbb{N}}$ is $k$-regular, from~\cite{allouche_shallit_2003}[Thm. 16.2.2], $(pL_w(n+1))_{n \in \mathbb{N}}$ is $k$-regular. Moreover, from~\cite{allouche_shallit_2003}[Thm. 16.2.1], $(\Delta pL_w(n))_{n \in \mathbb{N}}$ is $k$-regular. From Corollary~\ref{cor}, $(\Delta pL_w(n))_{n \in \mathbb{N}}$ is bounded, thus, it is $k$-automatic.\qed
\end{proof}

\begin{exa}
Let $t$ be the Thue-Morse word, the fixed point of the morphism $\mu$
sending $0$ to $01$ and $1$ to $10$.
Using the free software Walnut~\cite{mousavi2016automatic}, we can implement the algorithm in the
previous proof to find automata generating $(eL_t(n))_{n \in \mathbb{N}}$ and 
$(\Delta pL_t(n))_{n \in \mathbb{N}}$.
The sequence $(eL_t(n))_{n \in \mathbb{N}}$ is generated by the following automaton:
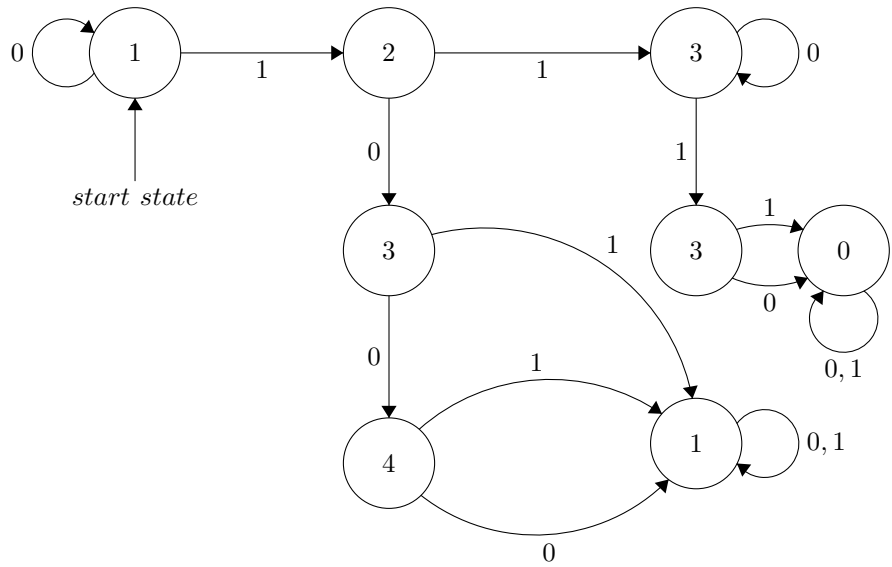
\begin{center}
\begin{tikzpicture}[scale=0.2]
\tikzstyle{every node}+=[inner sep=0pt]
\draw [black] (8.5,-3.2) circle (3);
\draw (8.5,-3.2) node {$1$};
\draw [black] (25.2,-3.2) circle (3);
\draw (25.2,-3.2) node {$2$};
\draw [black] (25.2,-16.3) circle (3);
\draw (25.2,-16.3) node {$3$};
\draw [black] (45.4,-3.2) circle (3);
\draw (45.4,-3.2) node {$3$};
\draw [black] (25.2,-30.4) circle (3);
\draw (25.2,-30.4) node {$4$};
\draw [black] (45.4,-29.1) circle (3);
\draw (45.4,-29.1) node {$1$};
\draw [black] (45.4,-16.3) circle (3);
\draw (45.4,-16.3) node {$3$};
\draw [black] (55.1,-16.3) circle (3);
\draw (55.1,-16.3) node {$0$};
\draw [black] (5.82,-4.523) arc (324:36:2.25);
\draw (1.25,-3.2) node [left] {$0$};
\fill [black] (5.82,-1.88) -- (5.47,-1) -- (4.88,-1.81);
\draw [black] (11.5,-3.2) -- (22.2,-3.2);
\fill [black] (22.2,-3.2) -- (21.4,-2.7) -- (21.4,-3.7);
\draw (16.85,-3.7) node [below] {$1$};
\draw [black] (25.2,-6.2) -- (25.2,-13.3);
\fill [black] (25.2,-13.3) -- (25.7,-12.5) -- (24.7,-12.5);
\draw (24.7,-9.75) node [left] {$0$};
\draw [black] (28.2,-3.2) -- (42.4,-3.2);
\fill [black] (42.4,-3.2) -- (41.6,-2.7) -- (41.6,-3.7);
\draw (35.3,-3.7) node [below] {$1$};
\draw [black] (25.2,-19.3) -- (25.2,-27.4);
\fill [black] (25.2,-27.4) -- (25.7,-26.6) -- (24.7,-26.6);
\draw (24.7,-23.35) node [left] {$0$};
\draw [black] (28.002,-15.246) arc (104.44491:10.83322:13.928);
\fill [black] (45.16,-26.12) -- (45.5,-25.24) -- (44.52,-25.42);
\draw (39.93,-16.47) node [above] {$1$};
\draw [black] (48.08,-1.877) arc (144:-144:2.25);
\draw (52.65,-3.2) node [right] {$0$};
\fill [black] (48.08,-4.52) -- (48.43,-5.4) -- (49.02,-4.59);
\draw [black] (45.4,-6.2) -- (45.4,-13.3);
\fill [black] (45.4,-13.3) -- (45.9,-12.5) -- (44.9,-12.5);
\draw (44.9,-9.75) node [left] {$1$};
\draw [black] (43.591,-31.484) arc (-44.24958:-128.38587:12.183);
\fill [black] (43.59,-31.48) -- (42.67,-31.71) -- (43.39,-32.41);
\draw (35.73,-35.7) node [below] {$0$};
\draw [black] (27.181,-28.156) arc (131.91506:55.44949:12.927);
\fill [black] (43.15,-27.13) -- (42.77,-26.26) -- (42.21,-27.09);
\draw (34.9,-24.32) node [above] {$1$};
\draw [black] (48.08,-27.777) arc (144:-144:2.25);
\draw (52.65,-29.1) node [right] {$0,1$};
\fill [black] (48.08,-30.42) -- (48.43,-31.3) -- (49.02,-30.49);
\draw [black] (48.037,-14.909) arc (106.67691:73.32309:7.713);
\fill [black] (52.46,-14.91) -- (51.84,-14.2) -- (51.55,-15.16);
\draw (50.25,-14.08) node [above] {$1$};
\draw [black] (52.75,-18.118) arc (-66.16759:-113.83241:6.188);
\fill [black] (52.75,-18.12) -- (51.82,-17.98) -- (52.22,-18.9);
\draw (50.25,-19.15) node [below] {$0$};
\draw [black] (56.423,-18.98) arc (54:-234:2.25);
\draw (55.1,-23.55) node [below] {$0,1$};
\fill [black] (53.78,-18.98) -- (52.9,-19.33) -- (53.71,-19.92);
\draw [black] (8.5,-11.7) -- (8.5,-6.2);
\draw (8.5,-12.2) node [below] {$start\mbox{ }state$};
\fill [black] (8.5,-6.2) -- (8,-7) -- (9,-7);
\end{tikzpicture}
\end{center}

Similarly, we obtain the following automaton generating $(\Delta pL_t(n))_{n \in \mathbb{N}}$:

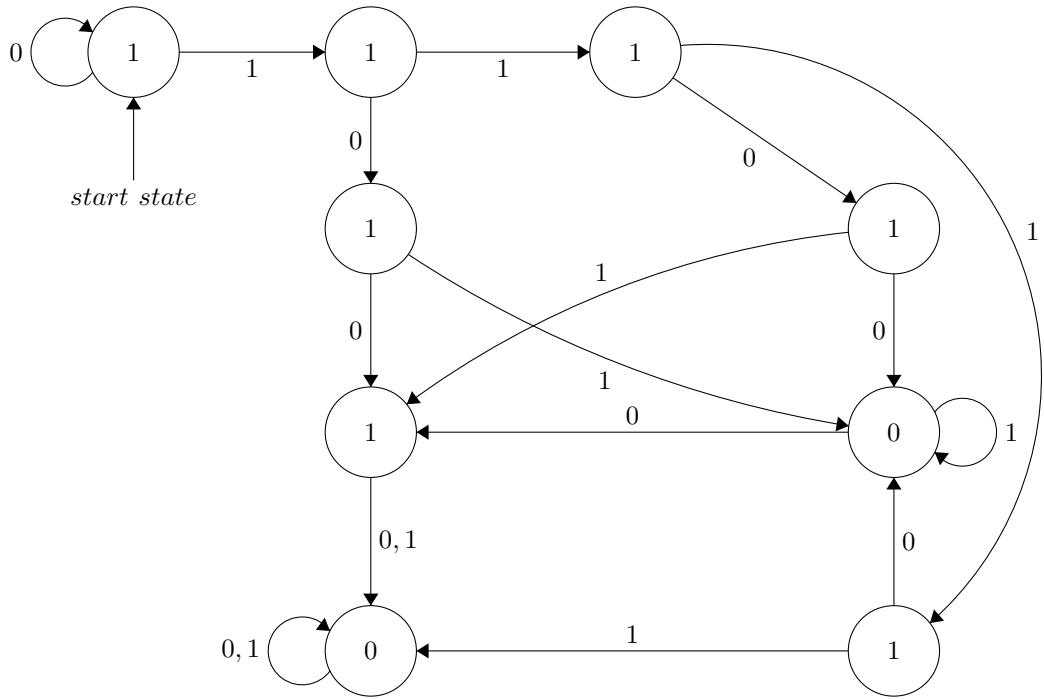
\begin{center}
\begin{tikzpicture}[scale=0.2]
\tikzstyle{every node}+=[inner sep=0pt]
\draw [black] (8.5,-3.2) circle (3);
\draw (8.5,-3.2) node {$1$};
\draw [black] (24.1,-3.2) circle (3);
\draw (24.1,-3.2) node {$1$};
\draw [black] (24.1,-14.9) circle (3);
\draw (24.1,-14.9) node {$1$};
\draw [black] (41.5,-3.2) circle (3);
\draw (41.5,-3.2) node {$1$};
\draw [black] (24.1,-28.4) circle (3);
\draw (24.1,-28.4) node {$1$};
\draw [black] (58.5,-28.4) circle (3);
\draw (58.5,-28.4) node {$0$};
\draw [black] (58.5,-42.9) circle (3);
\draw (58.5,-42.9) node {$1$};
\draw [black] (58.5,-14.9) circle (3);
\draw (58.5,-14.9) node {$1$};
\draw [black] (24.1,-42.9) circle (3);
\draw (24.1,-42.9) node {$0$};
\draw [black] (5.82,-4.523) arc (324:36:2.25);
\draw (1.25,-3.2) node [left] {$0$};
\fill [black] (5.82,-1.88) -- (5.47,-1) -- (4.88,-1.81);
\draw [black] (11.5,-3.2) -- (21.1,-3.2);
\fill [black] (21.1,-3.2) -- (20.3,-2.7) -- (20.3,-3.7);
\draw (16.3,-3.7) node [below] {$1$};
\draw [black] (24.1,-6.2) -- (24.1,-11.9);
\fill [black] (24.1,-11.9) -- (24.6,-11.1) -- (23.6,-11.1);
\draw (23.6,-9.05) node [left] {$0$};
\draw [black] (27.1,-3.2) -- (38.5,-3.2);
\fill [black] (38.5,-3.2) -- (37.7,-2.7) -- (37.7,-3.7);
\draw (32.8,-3.7) node [below] {$1$};
\draw [black] (24.1,-17.9) -- (24.1,-25.4);
\fill [black] (24.1,-25.4) -- (24.6,-24.6) -- (23.6,-24.6);
\draw (23.6,-21.65) node [left] {$0$};
\draw [black] (55.529,-27.982) arc (-99.17261:-123.68163:73.305);
\fill [black] (55.53,-27.98) -- (54.82,-27.36) -- (54.66,-28.35);
\draw (39.48,-24.37) node [below] {$1$};
\draw [black] (44.465,-2.757) arc (94.58563:-48.22329:21.981);
\fill [black] (60.87,-41.06) -- (61.8,-40.9) -- (61.13,-40.15);
\draw (67.16,-15.06) node [right] {$1$};
\draw [black] (61.18,-27.077) arc (144:-144:2.25);
\draw (65.75,-28.4) node [right] {$1$};
\fill [black] (61.18,-29.72) -- (61.53,-30.6) -- (62.12,-29.79);
\draw [black] (8.5,-11.7) -- (8.5,-6.2);
\draw (8.5,-12.2) node [below] {$start\mbox{ }state$};
\fill [black] (8.5,-6.2) -- (8,-7) -- (9,-7);
\draw [black] (43.97,-4.9) -- (56.03,-13.2);
\fill [black] (56.03,-13.2) -- (55.65,-12.33) -- (55.09,-13.16);
\draw (49,-9.55) node [below] {$0$};
\draw [black] (24.1,-31.4) -- (24.1,-39.9);
\fill [black] (24.1,-39.9) -- (24.6,-39.1) -- (23.6,-39.1);
\draw (24.6,-35.65) node [right] {$0,1$};
\draw [black] (55.5,-28.4) -- (27.1,-28.4);
\fill [black] (27.1,-28.4) -- (27.9,-28.9) -- (27.9,-27.9);
\draw (41.3,-27.9) node [above] {$0$};
\draw [black] (58.5,-17.9) -- (58.5,-25.4);
\fill [black] (58.5,-25.4) -- (59,-24.6) -- (58,-24.6);
\draw (58,-21.65) node [left] {$0$};
\draw [black] (26.458,-26.546) arc (126.72348:96.13076:59.151);
\fill [black] (26.46,-26.55) -- (27.4,-26.47) -- (26.8,-25.67);
\draw (39.27,-18.38) node [above] {$1$};
\draw [black] (58.5,-39.9) -- (58.5,-31.4);
\fill [black] (58.5,-31.4) -- (58,-32.2) -- (59,-32.2);
\draw (59,-35.65) node [right] {$0$};
\draw [black] (55.5,-42.9) -- (27.1,-42.9);
\fill [black] (27.1,-42.9) -- (27.9,-43.4) -- (27.9,-42.4);
\draw (41.3,-42.4) node [above] {$1$};
\draw [black] (21.42,-44.223) arc (-36:-324:2.25);
\draw (16.85,-42.9) node [left] {$0,1$};
\fill [black] (21.42,-41.58) -- (21.07,-40.7) -- (20.48,-41.51);
\end{tikzpicture}
\end{center}

\end{exa}

\section{Acknowledgement}

The author would like to thank Jason Bell and Jeffrey Shallit for a helpful discussion. The example proposed in Section Automaticity was provided by Jeffrey Shallit using Walnut. The author remarks that in a recent paper~\cite{aless2022lie}, Alessandro De Luca and Gabriele Fici reproved Theorem~\ref{Lie:1} by using Rauzy graphs. However, our approaches are merely different.


\bibliographystyle{splncs03}
\bibliography{biblio}
\end{document}